\documentclass[12pt]{amsart}

\usepackage{graphicx}
\usepackage{amssymb}
\usepackage{times}
\usepackage[dvipdfm]{color}
\newcommand{\C}{\Bbb C}

\newcommand{\Z}{\Bbb Z}
\newcommand{\tr}{\mathrm{tr}\,}
\newcommand{\D}{\Delta_{K,\rho}(t)}
\newcommand{\la}{\langle}
\newcommand{\ra}{\rangle}

\numberwithin{equation}{section}
\newtheorem{theorem}{Theorem}[section]

\newtheorem{lemma}[theorem]{Lemma}

\theoremstyle{definition}

\newtheorem{example}[theorem]{Example}
\newtheorem{remark}[theorem]{Remark}

\begin{document}

\title{On a conjecture of Dunfield, Friedl and Jackson}

\author{Takayuki Morifuji}


\address{Department of Mathematics, 
Hiyoshi Campus, Keio University, 
Yokohama 223-8521, Japan}

\email{morifuji@z8.keio.jp}

\maketitle

\begin{abstract}
In this short Note, 
we show that the twisted Alexander polynomial associated to a parabolic 
$SL(2,\C)$-representation detects genus and fibering of the twist knots. 
As a corollary, a conjecture of Dunfield, Friedl and Jackson is proved for 
the hyperbolic twist knots.
\end{abstract}

\section{Introduction}

Let $K$ be a knot in the $3$-sphere $S^3$ and 
denote its knot group by $G(K)$. That is, 
$G(K)=\pi_1E(K)$ where 
$E(K)$ is the knot exterior $S^3\backslash\text{int}(N(K))$ 
which is a compact $3$-manifold with torus boundary. 
For a nonabelian representation $\rho:G(K)\to SL(2,\C)$, 
the \textit{twisted Alexander polynomial}\/ $\D\in\C[t^{\pm1}]$ is defined 
up to multiplication by some $t^i$, with $i\in \Z$, 
see \cite{Lin01-1}, \cite{Wada94-1} and \cite{KM05-1} for details. 

If $K$ is a hyperbolic knot, 
namely the interior of $E(K)$ admits the complete hyperbolic metric 
with finite volume, there is a discrete faithful representation 
$\rho_0:G(K)\to SL(2,\C)$, 
which is called the \textit{holonomy representation}, 
corresponding to 
the hyperbolic structure. 

The \textit{hyperbolic torsion polynomial}\/ 
$\mathcal{T}_K(t)\in\C[t^{\pm1}]$ was defined in \cite{DFJ10-1} 
for hyperbolic knots as a suitable normalization of $\Delta_{K,\rho_0}(t)$. 
It is a symmetric polynomial in the sense that 
$\mathcal{T}_K(t^{-1})=\mathcal{T}_K(t)$, 
which seems to contain geometric information. 
In fact 
Dunfield, Friedl and Jackson conjectured in \cite{DFJ10-1} that 
$\mathcal{T}_K$ determines the genus $g(K)$ and moreover, 
the knot $K$ is fibered if and only if $\mathcal{T}_K$ is monic. 
 
They show in \cite{DFJ10-1} that the conjecture holds for all hyperbolic knots 
with at most 15 crossings. Our main theorem in this note is the following. 

\begin{theorem}\label{thm:main}
For all hyperbolic twist knots $K$ (see Figure 1) 
$\mathcal{T}_K$ determines the genus $g(K)$ and moreover, 
the knot $K$ is fibered if and only if 
$\mathcal{T}_K$ is monic. 
\end{theorem}

As far as we know, 
this is the first infinite family of knots for which 
the conjecture is verified. 
Since 
twist knots are 2-bridge knots (in particular alternating knots), 
their genus and fibering can 
be detected by the Alexander polynomial 
(see \cite{Crowell59-1}, \cite{Murasugi58-1}, \cite{Murasugi58-2}, \cite{Murasugi63-1}). 
However there seems to be no a priori reason 
that the same must be true for $\mathcal{T}_K$. 
See \cite[Section~7]{DFJ10-1}, \cite{KM12-1} 
for twisted Alexander polynomials 
and character varieties of knot groups. 

Recall that 
an $SL(2,\C)$-representation $\rho$ is called \textit{parabolic}\/ if 
the meridian of $G(K)$ is sent to a parabolic element of $SL(2,\C)$ and 
$\rho(G(K))$ is nonabelian. 
Since the holonomy representation of hyperbolic knots is parabolic, 
the above theorem is an immediate 
consequence of the following: 

\begin{theorem}\label{thm:parabolic}
Let $K$ be a twist knot and $\rho:G(K)\to SL(2,\C)$ 
a parabolic representation. 
Then $\D$ determines $g(K)$. 
Moreover $K$ is fibered if and only if $\D$ is monic. 
\end{theorem}

In the next section, 
we quickly review twisted Alexander polynomials of twist knots 
for parabolic $SL(2,\C)$-representations (see \cite[Sections~3,4]{Morifuji08-1} for details). 
The proof of Theorem~\ref{thm:parabolic} will be given in Section 3. 

\section{Twisted Alexander polynomials of twist knots}

Let $K=J(\pm2,p)$ be the twist knot ($p\in\Z$). 
It is known that 
$J(\pm2,2q+1)$ is equivalent to 
$J(\mp2,2q)$ 
and $J(\pm2,p)$ is the mirror image of 
$J(\mp2,-p)$. 
Hence 
we only consider the case where 
$K=J(2,2q)$ for $q\in\Z$ (see Figure 1). 
The knot $J(2,0)$ presents the trivial knot, 
so that we always assume $q\not=0$. 
The typical examples 
are 
the trefoil knot $J(2,2)$ and 
the figure eight knot $J(2,-2)$. 

The twist knots are alternating knots and have genus one. 
The Alexander polynomial of $K=J(2,2q)$ is given by 
$\Delta_K(t)=q-(2q-1)t+q t^2$. 
Furthermore, 
it is known (see \cite{Murasugi63-1}) that 
$J(2,2q)$ is fibered if and only if $|q|=1$. 
It is also known that $J(2,2q)$ is hyperbolic if 
$q\not\in\{0,1\}$. 

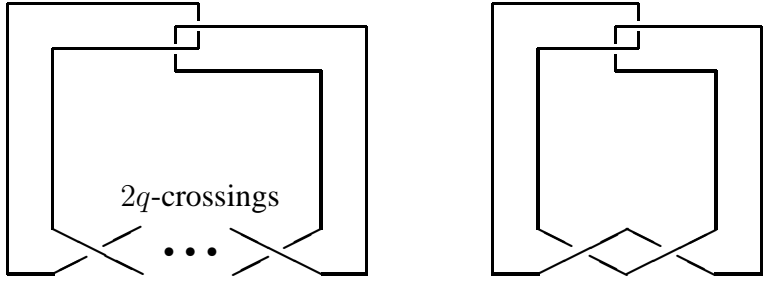
\begin{figure}
\setlength{\unitlength}{0.15mm}
\thicklines{
\begin{picture}(300,280)(120,-20)
\put(-19,0){\line(-1,0){41}}
\put(40,60){{$2q$-crossings}}
\put(78,15){{\tiny $\bullet\ \bullet\ \bullet$}}
\put(218,0){\line(1,0){41}}
\put(218,0){\line(-2,1){80}}
\put(218,40){\line(-2,-1){33}}
\put(171,15){\line(-2,-1){31}}
\put(-20,40){\line(2,-1){80}}
\put(-20,0){\line(2,1){30}}
\put(25,25){\line(2,1){33}}
\put(-60,0){\line(0,1){240}}
\put(-20,40){\line(0,1){160}}
\put(218,40){\line(0,1){140}}
\put(258,0){\line(0,1){220}}
\put(-60,240){\line(1,0){169}}
\put(89,220){\line(1,0){169}}
\put(109,240){\line(0,-1){15}}
\put(109,200){\line(0,1){15}}
\put(-20,200){\line(1,0){129}}
\put(89,180){\line(1,0){129}}
\put(89,180){\line(0,1){15}}
\put(89,220){\line(0,-1){15}}
\put(411,0){\line(-1,0){42}}
\put(489,40){\line(2,-1){30}}
\put(534,16){\line(2,-1){33}}
\put(489,0){\line(2,1){79}}
\put(567,0){\line(1,0){42}}
\put(410,40){\line(2,-1){30}}
\put(456,16){\line(2,-1){33}}
\put(410,0){\line(2,1){80}}
\put(370,0){\line(0,1){240}}
\put(410,40){\line(0,1){160}}
\put(568,40){\line(0,1){140}}
\put(608,0){\line(0,1){220}}
\put(370,240){\line(1,0){129}}
\put(479,220){\line(1,0){129}}
\put(499,240){\line(0,-1){15}}
\put(499,200){\line(0,1){15}}
\put(410,200){\line(1,0){89}}
\put(479,180){\line(1,0){89}}
\put(479,180){\line(0,1){15}}
\put(479,220){\line(0,-1){15}}
\end{picture}
}
\caption{$K=J(2,2q)$ and the figure eight knot $J(2,-2)$}
\end{figure}

The knot group $G(J(2,2q))$ has the presentation: 
$$
G(J(2,2q))
=
\la
x,y~|~w^qx=yw^q 
\ra,
$$
where 
$w=[y,x^{-1}]$. 
Suppose that $\rho:G(J(2,2q))\to SL(2,\C)$ is a parabolic representation. 
After conjugating, if necessary, we may assume that 
$$
\rho(x)
=\begin{pmatrix}
1 & 1 \\ 
0 & 1 
\end{pmatrix}\quad
\text{and}\quad
\rho(y)
=\begin{pmatrix} 
1 & 0 \\
-u & 1
\end{pmatrix}.
$$
Let $\rho(w^q)=(a_{ij}(u))$ and write $\phi_q(u)=a_{11}(u)$. 
It is known that 
$\rho$ defines a group representation when 
$u$ satisfies $\phi_q(u)=0$ (see \cite[Theorem~2]{Riley72-1}). 
We call $\phi_q(u)$ the 
\textit{Riley polynomial}\/ of the twist knot $J(2,2q)$. 
By \cite[Proposition~3.1]{Morifuji08-1}, 
$\phi_q(u)$ has an explicit formula 
$$
\phi_q(u)
=\left(1-u\right)\frac{\lambda_+^q-\lambda_-^q}{\lambda_+-\lambda_-}
-\frac{\lambda_+^{q-1}-\lambda_-^{q-1}}{\lambda_+-\lambda_-},
$$
where 
$$
\lambda_\pm(u)
=
\frac
{{u^2+ 2}
\pm
\sqrt{u^4+4u^2}}
{2}
$$
denote the eigenvalues of the matrix $\rho(w)$. 
Of course, 
the holonomy representation $\rho_0$ corresponds to one 
of the roots of $\phi_q(u)=0$. 

\begin{lemma}\label{lem:Riley}
The Riley polynomial $\phi_q(u)$ satisfies the following:
\begin{enumerate}
\item
The highest coefficient of $\phi_q(u)$ is $\pm1$.
\item
$\phi_q(u)\in \Z[u]$ is irreducible. 
\item
$\deg\phi_q(u)=2q-1~(q>0)$ or $2|q|~(q<0)$.
\end{enumerate}
\end{lemma}

\begin{proof}
(1) See \cite[Theorem~2]{Riley72-1}. 
(2), (3) See \cite[Theorem~1]{HS01-1}. 
\end{proof}

\begin{example}\label{ex:Riley}
We can easily check that 
$\phi_1(u)=1-u,~
\phi_{-1}(u)=1+u+u^2,~
\phi_2(u)=1-2u+u^2-u^3$, 
$\phi_{-2}(u)=1+2u+3u^2+u^3+u^4$ 
and 
$\phi_3(u)=1-3u+3u^2-4u^3+u^4-u^5$. 
\end{example}

\begin{lemma}\label{lem:formula}
For a parabolic representation $\rho:G(K)\to SL(2,\C)$ of $K=J(2,2q)$, 
the twisted Alexander polynomial 
$\Delta_{K,\rho}(t)$ is given by 
$$
\Delta_{K,\rho}(t)
~ =\, \alpha\beta  
+
\left\{
\alpha+\beta-2\alpha\beta
+\frac{\lambda_+-\lambda_-}{2+\lambda_++\lambda_-}(\alpha-\beta)
\right\}t
+\alpha\beta t^2,
$$
where
$\alpha=1+\lambda_++\lambda_+^2+\cdots+\lambda_+^{q-1}$ 
and 
$\beta=1+\lambda_-+\lambda_-^2+\cdots+\lambda_-^{q-1}$.
\end{lemma}

\begin{proof}
We only have to put $s=1$ in the formula of \cite[Theorem~4.1]{Morifuji08-1}. 
\end{proof}

\begin{example}\label{ex:TAP}
For $K=J(2,2)$, 
there is just one parabolic representation up to conjugation 
and 
we have $\D=1+t^2$. 
Similarly 
we obtain $\D=1-4t+t^2$ for any parabolic representation of $K=J(2,-2)$. 
\end{example}

In general, 
the degree of the twisted Alexander polynomial gives a 
lower bound for the knot genus $g(K)$. 
In fact, 
for every nonabelian representation $\rho:G(K)\to SL(2,\C)$, 
the following inequality holds (see \cite{FK05-1}):
\begin{equation}\label{eq:bound}
4g(K)-2\geq\deg\D. 
\end{equation}
When the equality holds in (\ref{eq:bound}), 
we say $\D$ determines the knot genus. 
For a fibered knot $K$, 
it is known that $\D$ determines $g(K)$ and is a monic polynomial 
(see \cite{Cha03-1}, \cite{FK05-1}, \cite{FV11-1}, \cite{GKM05-1}, \cite{KM05-1}). 

\section{Proof of Theorem~\ref{thm:parabolic}}

First 
we denote the highest coefficient of $\D$ in Lemma~\ref{lem:formula} 
by $\gamma_q(u)$, 
namely $\gamma_q(u)=\alpha\beta$. Moreover we put 
$\tau_q(u)=\tr\rho(w^q)=\lambda_+^q+\lambda_-^q$. 
By \cite[Corollary~4.3]{Morifuji08-1}, 
$\tau_q(u)=\tau_{-q}(u)$ is a monic polynomial in $\Z[u]$ and 
$\deg\tau_q(u)=2|q|$. 

\begin{example}\label{ex:trace}
Since $\tau_1(u)=u^2+2$, 
we obtain 
$\tau_{\pm2}(u)={\tau_1}^2-2=u^4+4u^2+2$ and 
$\tau_{\pm3}(u)={\tau_1}^3-3\tau_1=u^6+6u^4+9u^2+2$. 
\end{example}

Now an easy calculation shows that 
\begin{align*}
\gamma_q(u)
&=
(1+\lambda_++\lambda_+^2+\cdots+\lambda_+^{q-1})
(1+\lambda_-+\lambda_-^2+\cdots+\lambda_-^{q-1})\\
&=
\tau_{q-1}(u)+(\text{some polynomial in $\tau_1,\ldots,\tau_{q-2}$}).
\end{align*}
Thus 
we have $\deg\gamma_q(u)=2|q|-2$. 
By Lemma~\ref{lem:Riley} (1), (2), 
if $\gamma_q(u)=0$ for a complex number $u$ satisfying $\phi_q(u)=0$, 
then the Riley polynomial $\phi_q(u)$ divides $\gamma_q(u)$. 
But this contradicts the fact that 
\begin{align*}
\deg\phi_q(u)
&=
2|q|-\max\{\text{sign}(q),0\}\\
&>
2|q|-2=\deg\gamma_q(u).
\end{align*}
Hence 
$\gamma_q(u)$ never vanishes for the parabolic representations. 
Thus 
$\deg\D=2$ and hence it determines the genus. 

A similar argument applied to $\gamma_q(u)-1$ shows that 
$\D$ is not a monic polynomial for the nonfibered twist knot 
$K=J(2,2q)$ with $|q|>1$. 
This completes the proof of Theorem~\ref{thm:parabolic}. 

\begin{remark}\label{rmk:2-bridge}
For the 3830 nonfibered $2$-bridge knots $K(a,b)$ with $b<a\leq287$, 
Dunfield, Friedl and Jackson numerically compute 
the twisted Alexander polynomials for the parabolic representations. 
In fact, 
it is shown in \cite[Section~7.6]{DFJ10-1} that 
$\D$ is nonmonic and determines the knot genus in every case. 
\end{remark}

\begin{remark}\label{rmk:2nd}
Let $\delta_q(u)$ be the second coefficient of $\D$ 
in Lemma~\ref{lem:formula}. 
As we saw in Example~\ref{ex:TAP}, 
$\delta_{\pm1}(u)$ are integers for the fibered twist knots 
$J(2,\pm2)$. 
It is not so hard to show that 
$\delta_q(u)\in \Z[u]$ and 
$\deg\delta_q(u)=2q-4$ for $q>1$. 
Therefore 
we can conclude that $\delta_q(u)$ with $q>2$ 
is not a rational number 
for the parabolic representations. 
On the other hand, 
we have $\D=(u^2+4)-4t+(u^2+4)t^2$ for the hyperbolic twist knot $K=J(2,4)$. 
In particular, 
we see that 
the second coefficient $\delta_{2}(u)$ is an integer 
for the holonomy representation, 
although $J(2,4)$ is nonfibered 
(see \cite[Section~6.5]{DFJ10-1}).  
\end{remark}

\noindent
\textit{Acknowledgements}. 
The author would like to thank 
Stefan Friedl, Teruaki Kitano, Takahiro Kitayama and 
Yoshikazu Yamaguchi for helpful comments. 
The author also would like to thank the referee 
for useful comments, which helped to improve the Note. 
This research was partially supported by 
Grant-in-Aid for Scientific Research (No.\,23540076), 
the Ministry of Education, Culture, Sports, Science 
and Technology, Japan. 


\end{document}